\documentclass[12pt]{amsart}

\usepackage[utf8]{inputenc}
\usepackage[english]{babel}

\usepackage[letterpaper,top=2cm,bottom=2cm,left=3cm,right=3cm,marginparwidth=1.75cm]{geometry}

\usepackage{amsmath}
\usepackage{graphicx}
\usepackage[colorlinks=true, allcolors=blue]{hyperref}

\usepackage{amssymb} 
\usepackage{comment}
\usepackage{indentfirst}
\usepackage{csquotes}
\usepackage{theoremref}
\usepackage{hyperref}
\usepackage{geometry}

\hypersetup{hidelinks, colorlinks=true, allcolors=cyan, pdfstartview=Fit, breaklinks=true}

\usepackage{amsthm}

\theoremstyle{plain}
\newtheorem{theorem}{Theorem}
\numberwithin{theorem}{section}
\newtheorem{corollary}[theorem]{Corollary}
\newtheorem*{corollary*}{Corollary}
\newtheorem*{Example*}{Example}

\newtheorem{lemma}[theorem]{Lemma}

\theoremstyle{definition}

\newtheorem*{def*}{Definition}
\newtheorem*{theorem*}{Theorem}

\newtheorem*{definition*}{Definition}

\theoremstyle{remark}

\newcommand{\bracket}[1]{\left( #1 \right)}

\newcommand{\notmodulo}[3]{#1\not\equiv#2\ \bracket{\mathrm{mod}\ #3}}
\newcommand{\modulo}[3]{#1\equiv#2\ \bracket{\mathrm{mod}\ #3}}

\numberwithin{equation}{section}

\title{\textbf{Fibonacci-like property of partition function}}
\author{QI-YANG ZHENG}
\date{} 

\address{Department of Mathematics, Sun Yat-sen University(Zhuhai Campus), Zhuhai}
\email{zhengqy29@mail2.sysu.edu.cn}

\begin{document}
\maketitle

\begin{abstract}
The main result of the paper is the Fibonacci-like property of the partition function. The partition function $p(n)$ has a property: $p(n) \leq p(n-1) + p(n-2)$. Our result shows that if we impose certain restrictions on the partition, then the inequality becomes an equality. Furthermore, we extend this result to cases with a greater number of summands.
\end{abstract}


\section{Introduction}

\noindent
In number theory and combinatorics, a partition of a non-negative integer $n$, also called an integer partition, is a way of representing $n$ as a sum of positive integers. Two sums that differ only in the order of their summands are considered to be the same partition. We denote the partition function by $p(n)$.

One can employ elementary methods to demonstrate that $p(n) \leq p(n-1) + p(n-2)$ for $n \geq 2$ (see \cite[(3.8)]{andrews2004integer}). Our result provides an explicit formulation of this equation when certain restrictions are imposed. The theorem is outlined below:

\begin{theorem}
    \label{Fibonacci-like}
    For $n\geq2$, we have
    \begin{equation}
        \notag
        \abovedisplayskip=1em
        \begin{aligned}
            &\ \ \ \ p(n\mid\notmodulo{\mathrm{parts}}{12,15,27}{27}) \\
            &=p(n-1\mid\notmodulo{\mathrm{parts}}{6,21,27}{27}) \\
            &+p(n-2\mid\notmodulo{\mathrm{parts}}{3,24,27}{27}). \\
        \end{aligned}
        \belowdisplayskip=1em
    \end{equation}
\end{theorem}

\noindent
For example, let $n = 29$. We find that $p(27) = 3010$, $p(28) = 3718$, and $p(29) = 4565$. However,
\begin{equation}
    \notag
    \abovedisplayskip=1em
    \begin{aligned}
        p(29\mid\notmodulo{\mathrm{parts}}{12,15,27}{27})=4133, \\
        p(28\mid\notmodulo{\mathrm{parts}}{6,21,27}{27})=2701, \\
        p(27\mid\notmodulo{\mathrm{parts}}{3,24,27}{27})=1432. \\
    \end{aligned}
    \belowdisplayskip=1em
\end{equation}

\noindent
In fact, we prove something more. Recall that the partition function possesses a recurrence formula,
\begin{equation}
    \notag
    \abovedisplayskip=1em
    \begin{aligned}
        p(n)&=\sum_{k\in\mathbb Z\backslash\{0\}}(-1)^{k+1}p(n-k(3k-1)/2)\\
        &=p(n-1)+p(n-2)-p(n-5)-p(n-7)+\cdots
    \end{aligned}
    \belowdisplayskip=1em
\end{equation}

\noindent
Our result shows that we can truncate the formula at any even position, provided certain restrictions are applied to the partition. The precise statement is as follows:

\begin{theorem}
    \label{general case}
    For every integer $m\geq1$ and $n\geq m(3m+1)/2$, we have
    \begin{equation}
        \notag
        \abovedisplayskip=1em
        \begin{aligned}
            0&=p\left(n\mid{\mathrm{parts}}\not\equiv{0,(2m+1)(3m+1),(2m+1)(3m+2)}\right)+\sum_{i=1}^m(-1)^i \\
            &\left.\left[p\left(n-\frac{i(3i-1)}{2}\,\right|\right.{\mathrm{parts}}\not\equiv{0,(2m+1)(3m-3i+2),(2m+1)(3m+3i+1)}\right) \\
            &\left.\left.+p\left(n-\frac{i(3i+1)}{2}\,\right|\left.{\mathrm{parts}}\not\equiv{0,(2m+1)(3m-3i+1),(2m+1)(3m+3i+2)}\right. \right)\right],
        \end{aligned}
        \belowdisplayskip=1em
    \end{equation}

    \noindent
    where all three congruences are taken modulo $3(2m+1)^2$.
\end{theorem}

\noindent
In fact, if we agree that $p(n) = 0$ for $n < 0$, then the theorem is satisfied for all $n$.

Note that Theorem \ref{Fibonacci-like} corresponds to the special case $m=1$. Additionally, the case $m=0$ is trivial. If we set $m=2$, then we acquire

\begin{corollary}
    For $n\geq7$, we have
    \begin{equation}
        \notag
        \abovedisplayskip=1em
        \begin{aligned}
            &\ \ \ \ p(n\mid\notmodulo{\mathrm{parts}}{35,40,75}{75}) \\
            &=p(n-1\mid\notmodulo{\mathrm{parts}}{25,50,75}{75}) \\
            &+p(n-2\mid\notmodulo{\mathrm{parts}}{20,55,75}{75}) \\
            &-p(n-5\mid\notmodulo{\mathrm{parts}}{10,65,75}{75}) \\
            &-p(n-7\mid\notmodulo{\mathrm{parts}}{5,70,75}{75}). \\
        \end{aligned}
        \belowdisplayskip=1em
    \end{equation}
\end{corollary}

~

\section{Proofs of theorems}

\noindent
As mentioned before, it is sufficient to prove Theorem \ref{general case}.

\begin{proof}[Proof of Theorem \ref{general case}]
    First, we recall the well-known Euler's Pentagonal Number Theorem
    \begin{equation}
        \notag
        \abovedisplayskip=1em
        \prod_{n=1}^\infty(1-q^n)=\sum_{n=-\infty}^\infty(-1)^nq^{\frac{n(3n+1)}{2}}.
        \belowdisplayskip=1em
    \end{equation}

    \noindent
    Substitute $q$ with $q^{1/(2m+1)}$, resulting in
    \begin{equation}
        \notag
        \abovedisplayskip=1em
        \prod_{n=1}^\infty(1-q^{\frac{n}{2m+1}})=\sum_{n=-\infty}^\infty(-1)^nq^{\frac{n(3n+1)}{2(2m+1)}}.
        \belowdisplayskip=1em
    \end{equation}

    \noindent
    Now we divide the summation according to residue classes modulo $2m+1$,
    \begin{equation}
        \notag
        \abovedisplayskip=1em
        \begin{aligned}
            \prod_{n=1}^\infty(1-q^{\frac{n}{2m+1}})&=\sum_{n=-\infty}^\infty(-1)^nq^{\frac{n(3n+1)}{2(2m+1)}}\\
            &=\sum_{i=-m}^m\sum_{\genfrac{}{}{0pt}{}{n=-\infty}{\modulo{n}{i}{2m+1}}}^\infty(-1)^nq^{\frac{n(3n+1)}{2(2m+1)}}.
        \end{aligned}
        \belowdisplayskip=1em
    \end{equation}

    \noindent
    For each value of $i$, we have
    \begin{equation}
        \notag
        \abovedisplayskip=1em
        \begin{aligned}
            \sum_{\genfrac{}{}{0pt}{}{n=-\infty}{\modulo{n}{i}{2m+1}}}^\infty(-1)^nq^{\frac{n(3n+1)}{2(2m+1)}}&=
            \sum_{n=-\infty}^\infty(-1)^{(2m+1)n+i}q^{\frac{[(2m+1)n+i][3(2m+1)n+3i+1]}{2(2m+1)}} \\
            &=(-1)^iq^{\frac{i(3i+1)}{2(2m+1)}}\sum_{n=-\infty}^\infty(-1)^nq^{\frac{(2m+1)3n^2+n(6i+1)}{2}} \\
            &=(-1)^iq^{\frac{i(3i+1)}{2(2m+1)}}\sum_{n=-\infty}^\infty(-1)^nq^{3(2m+1)\frac{n(n+1)}{2}-(3m-3i+1)n}.
        \end{aligned}
        \belowdisplayskip=1em
    \end{equation}

    \noindent
    We now introduce a lemma to address this type of summation (cf. \cite[Corollary 2.9]{andrews1998theory}).
    
    \begin{lemma}
        For $|q|<1$,
        \begin{equation}
            \notag
            \abovedisplayskip=1em
            \begin{aligned}
                &\ \ \sum_{n=-\infty}^\infty(-1)^nq^{(2k+1)n(n+1)/2-in}\\
                &=\prod_{n=0}^\infty(1-q^{(2k+1)(n+1)})(1-q^{(2k+1)n+i})(1-q^{(2k+1)(n+1)-i}).
            \end{aligned}
            \belowdisplayskip=1em
        \end{equation}
    \end{lemma}

    \noindent
    Therefore,
    \begin{equation}
        \notag
        \abovedisplayskip=1em
        \begin{aligned}
            &\ \ \sum_{\genfrac{}{}{0pt}{}{n=-\infty}{\modulo{n}{i}{2m+1}}}^\infty(-1)^nq^{\frac{n(3n+1)}{2(2m+1)}}\\
            &=(-1)^iq^{\frac{i(3i+1)}{2(2m+1)}}\prod_{n=0}^\infty(1-q^{3(2m+1)(n+1)})(1-q^{3(2m+1)n+3m-3i+1})(1-q^{3(2m+1)n+3m+3i+2}).
        \end{aligned}
        \belowdisplayskip=1em
    \end{equation}

    \noindent
    It is worth noting that $0 < 3m - 3i + 1, 3m + 3i + 2 < 3(2m+1)$ for all $-m \leq i \leq m$.

    Next, we substitute $i$ with $-i$, resulting in
    \begin{equation}
        \notag
        \abovedisplayskip=1em
        \begin{aligned}
            &\ \ \sum_{\genfrac{}{}{0pt}{}{n=-\infty}{\modulo{n}{-i}{2m+1}}}^\infty(-1)^nq^{\frac{n(3n+1)}{2(2m+1)}}\\
            &=(-1)^iq^{\frac{i(3i-1)}{2(2m+1)}}\prod_{n=0}^\infty(1-q^{3(2m+1)(n+1)})(1-q^{3(2m+1)n+3m-3i+2})(1-q^{3(2m+1)n+3m+3i+1}).
        \end{aligned}
        \belowdisplayskip=1em
    \end{equation}

    \noindent
    Hence
    \begin{equation}
        \notag
        \abovedisplayskip=1em
        \begin{aligned}
            &\ \ \ \sum_{i=-m}^m\sum_{\genfrac{}{}{0pt}{}{n=-\infty}{\modulo{n}{i}{2m+1}}}^\infty(-1)^nq^{\frac{n(3n+1)}{2(2m+1)}} \\
            &=\left[\sum_{\genfrac{}{}{0pt}{}{n=-\infty}{\modulo{n}{0}{2m+1}}}^\infty
            +\sum_{i=1}^m\bracket{\sum_{\genfrac{}{}{0pt}{}{n=-\infty}{\modulo{n}{-i}{2m+1}}}^\infty+\sum_{\genfrac{}{}{0pt}{}{n=-\infty}{\modulo{n}{i}{2m+1}}}^\infty}\right](-1)^nq^{\frac{n(3n+1)}{2(2m+1)}}.
        \end{aligned}
        \belowdisplayskip=1em
    \end{equation}

    \noindent
    Now, we use the product formulas obtained earlier to yield the following:
    \begin{equation}
        \notag
        \abovedisplayskip=1em
        \begin{aligned}
            &\ \ \ \,\prod_{n=1}^\infty(1-q^{\frac{n}{2m+1}}) \\
            &=\prod_{n=0}^\infty(1-q^{3(2m+1)(n+1)})(1-q^{3(2m+1)n+3m+1})(1-q^{3(2m+1)n+3m+2})+\sum_{i=1}^m(-1)^i \\
            &\left[q^{\frac{i(3i-1)}{2(2m+1)}}\prod_{n=0}^\infty(1-q^{3(2m+1)(n+1)})(1-q^{3(2m+1)n+3m-3i+2})(1-q^{3(2m+1)n+3m+3i+1})\right. \\
            &+\left.q^{\frac{i(3i+1)}{2(2m+1)}}\prod_{n=0}^\infty(1-q^{3(2m+1)(n+1)})(1-q^{3(2m+1)n+3m-3i+1})(1-q^{3(2m+1)n+3m+3i+2})\right].
        \end{aligned}
        \belowdisplayskip=1em
    \end{equation}

    \noindent
    Then, substitute $q$ with $q^{2m+1}$, resulting in
    \begin{equation}
        \notag
        \abovedisplayskip=1em
        \begin{aligned}
            &\ \ \ \ \prod_{n=1}^\infty(1-q^n)\\
            &=\prod_{n=1}^\infty(1-q^{3(2m+1)^2(n+1)})(1-q^{3(2m+1)^2n+(2m+1)(3m+1)})(1-q^{3(2m+1)^2n+(2m+1)(3m+2)}) \\
            &+\sum_{i=1}^m(-1)^i\left[q^{\frac{i(3i-1)}{2}}\prod_{n=0}^\infty(1-q^{3(2m+1)^2(n+1)})(1-q^{3(2m+1)^2n+(2m+1)(3m-3i+2)})\right.\\
            &\ \ \ \ \ \ \ \ \ \ \ \ \ \ \ \ \ \ \ \ \ \ \ \ \ \ \ \ \ \ \ \ \ \ \ \ \ \ \ \ \ \ \ \ \ \ \ \ \ \ \ \ \ \ \ (1-q^{3(2m+1)^2n+(2m+1)(3m+3i+1)})\\
            &+q^{\frac{i(3i+1)}{2}}\prod_{n=0}^\infty(1-q^{3(2m+1)^2(n+1)})(1-q^{3(2m+1)^2n+(2m+1)(3m-3i+1)})\\
            &\left.\phantom{\frac11}\ \ \ \ \ \ \ \ \ \ \ \ \ \ \ \ \ \ \ \ \ \ \ \ \ \ \ \ \ \ \ \ \ \ \ \ \ \ (1-q^{3(2m+1)^2n+(2m+1)(3m+3i+2)})\right].
        \end{aligned}
        \belowdisplayskip=1em
    \end{equation}

    \noindent
    Divide both sides by $\prod_{n=1}^\infty(1-q^n)$ and compare the coefficients on both sides, which yields the desired result.
\end{proof}

\section{Musings}

\noindent
Is it possible to find a combinatorial explanation for Theorem \ref{Fibonacci-like}?

~


\end{document}